\documentclass[twoside, 12pt]{article}
\usepackage{latexsym}
\usepackage{amssymb}
\usepackage{amsmath}
\makeatletter
\DeclareRobustCommand{\qed}{%
  \ifmmode 
  \else \leavevmode\unskip\penalty9999 \hbox{}\nobreak\hfill
  \fi
  \quad\hbox{\qedsymbol}}
\newcommand{\openbox}{\leavevmode
  \hbox to.77778em{%
  \hfil\vrule
  \vbox to.675em{\hrule width.6em\vfil\hrule}%
  \vrule\hfil}}
\newcommand{\qedsymbol}{\openbox}
\newenvironment{proof}[1][\proofname]{\par
  \normalfont
  \topsep6\p@\@plus6\p@ \trivlist
  \item[\hskip\labelsep\itshape
    #1\@addpunct{.}]\ignorespaces
}{%
  \qed\endtrivlist
}
\newcommand{\proofname}{Proof}
\makeatother
\font\authorfont=cmti12

\newtheorem{lemma}{Lemma}[section]
\newtheorem{df}[lemma]{Definition}
\newtheorem{thm}[lemma]{Theorem}

\newtheorem{prop}[lemma]{Proposition}

\newcommand{\nn}{\mathbb{N}}
\newcommand{\rr}{\mathbb{R}}

\def\x{\times}
\def\ox{\otimes}
\def\I{{\mathcal I}}
\def\H{{\mathbb H}}
\def\HH{\H_0}
\def\L{{\mathcal L}}

\def\P{\text{P}}
\def\esp#1{\text{E}[ #1 ]}
\def\D{{\mathbb D}}
\def\Hol{{\mathcal H}}
\def\Id{{\text Id}}
\def\n{\nabla}
\begin{document}
\noindent
\begin{center}{ \Large\bf Stochastic Volterra Equations\\[2mm]
    with Singular Kernels}
\end{center}
\vskip 2pc
\begin{center} \authorfont L. Coutin and L. Decreusefond
\end{center}

\vspace{5mm}
\noindent

\thispagestyle{empty}
\section{Introduction}
Motivated by the potential applications to the fractional Brownian
motion (cf. \cite{decreusefond96.2}), we study Volterra stochastic
differential of the form~:
\begin{equation}
X_t =  x+ \int_0^tK(t,s)b(s,X_s)ds + \int_0^tK(t,s) \sigma(s,X_s)
\,dB_s ,\tag{E} \label{eq:sdefbm}
\end{equation}
where $(B_s, \, s\in [0,1])$ is a one-dimensional standard Brownian
motion and $(K(t,s), \, t,s \in [0,1])$ is a deterministic kernel
whose properties will be precised below but for which we don't assume
any boundedness property.

Actually, when $\sigma$ is a constant and $K$ is given by
(\ref{defdekh}), we obtain~:
\begin{equation*}
X_t =  x+ \int_0^tK(t,s)b(s,X_s)ds +\sigma W_t^H,
\end{equation*}
where $W^H$ is the fractional Brownian motion of Hurst parameter $H$
-- see the example below.  In this particular case, the main feature
is that $K$ is highly singular as a kernel but the integral map 
canonically associated to it, i.e.,
 \begin{displaymath}
   Kf(t)=\int_0^t  K(t,s)f(s)\, ds,
 \end{displaymath}
 is a regularizing operator. That explains why we work as much as
 possible with the properties of the map $K$ and not with those of the
 kernel $K(t,s).$ The problem is then in
 the treatment of the stochastic integral. Actually, one of the main
 difficulties is to control the (H\"older) regularity, with respect to
 $t,$ of the stochastic integral in the right-hand-side of (E). This
 has been the object of a previous paper \cite{decreusefond98.3}, the
 hypothesis of which we simplify here. With the result obtained in
 that paper, the proof of existence and uniqueness of the solution of
 (E) is achieved as usual by a fixed point technique.
 
 However, some other problems arise when we study the Gross-Sobolev
 regularity of the solution of (E)  and the expression of its derivative.
 They are of two sorts : on one hand the singularity of the kernel and
 on the other hand the fact that the Gross-Sobolev derivative of the
 solution is the solution of a linear but time dependent
 stochastic differential equation.  We eventually give a somewhat
 explicit expression for the Gross-Sobolev derivative of the solution -- in this part, the approach owes much to the ideas
 developed in \cite{hirsch88}.
 
 Note that the specific form of the drift ensures both a symmetric
 role to $b$ and $\sigma$ and the existence of weak solutions to
 (\ref{eq:sdefbm}) -- cf.  \cite{decreusefond97.2} for the application
 of this notion to the non-linear filtering theory with fractional
 Gaussian noise. For other stochastic differential equations related
 to the fractional Brownian motion, we refer to
 \cite{lin96,lyons94,zaehle99}.
 
 The equations we have to deal with are of Volterra type but our work
 does not seem to be subsumed by previous articles on this subject (see
 for instance \cite{lewin94,pardoux90,protter85}) because our kernel
 is weakly regular and we are looking for classical solutions and not
 distribution-valued ones. Our work is only done in one dimension but
 it is straightforward to extend it to higher dimensions.

This paper is organized as follows: in the next section, we show how
hypothesis (A) is sufficient to entail those of the main theorem of
\cite{decreusefond98.3}. As an example, we address the case of the
fractional Brownian motion. In Section 3, we prove the main theorem of
existence and uniqueness of the solution of (E). In Section 4, we
prove that under an extra  boundedness assumption on $\sigma,$ this solution
is Gross-Sobolev differentiable and we give an integral representation of it.  
\section{Preliminaries}
\label{sec:preliminaries}

Consider a measurable kernel $(K(t,s), s,t\in [0,1])$ and denote also
by $K$ the (formal) linear map~:
\begin{displaymath}
  Kf(t)=\int_0^t K(t,s) f(s)\, ds.
\end{displaymath}
\noindent \textbf{\textsc{Hypothesis} (A)--}\textsl{
  We assume once for all that there exist $\gamma >0 $ such that $K$
  is a continuous both from $\L^1([0,t])$ into $\I_{\gamma,2}([0,t])$ and from
  $\L^2([0,t])$ into $\I_{\gamma+1/2,2}([0,t])$ for any $t\in [0,1].$}

\begin{thm}
\label{thm:regularite}
Let $u$ be an adapted process belonging to $L^{r}(\Omega\x [0,1],
\P\ox dt). $ If $r \ge 2,$ the family of random variables
$(M_t(u)=\int_0^t K(t,s) u_s \, dB_s, \, t\in [0,1])$ has a version
whose trajectories belong almost surely to $\I_{\gamma-\epsilon,r}$
for any $\epsilon\in (0,\gamma).$ Consequently, if $\gamma > 1/r,$ the
sample-paths are almost surely $(\gamma\! -\!1/r-\!\epsilon)$-H\"older
continuous for any $\epsilon\in (0,\gamma).$ Moreover, the following maximal
inequality holds~:
  \begin{equation}
    \label{eq:1}
    \lVert \, M(u) \rVert_{L^r(\Omega;\I_{\gamma-\epsilon,r}([0,t]))}\le \, c \lVert u
    \rVert_{L^r(\Omega \x [0,t])}.
  \end{equation}
\end{thm}
\begin{proof}
  It is sufficient to prove that hypothesis I to III of
  \cite{decreusefond98.3} are satisfied. Let $\delta=\gamma-\epsilon,
  $ we have to show that
\begin{enumerate}
\item $K$ is continuous from $\L^2$ into $\mathcal B,$ the space of
  bounded functions on $[0,1].$
\item $K$ and $K_\delta=I_{0^+}^{-\delta}\circ K $ are Hilbert-Schmidt
  from $\L^2$ into itself,
\end{enumerate}
Since $\gamma > 0,$ it is well known (see section A of the appendix)
that $\I_{\gamma+1/2,1/2}\subset \Hol_{\gamma}\subset {\mathcal B}$
and the first point follows.

Since $\gamma >0, $ the embedding of $\I_{\gamma+1/2,2}$ in $\L^2$ is
Hilbert-Schmidt  (see \cite{schwartz70}) and so is  $K$
from $\L^2$ into itself.
Moreover, $K_\delta$ is continuous from $\L^2$ into
$\I_{1/2+\epsilon,2}$ and the embedding of $\I_{1/2+\epsilon,2}$ in
$\L^2$ is Hilbert-Schmidt, hence $K_\delta$ is Hilbert-Schmidt from $\L^2$
into itself.
\end{proof}

As an example (in fact the motivating one), consider $K=K_H$ the
kernel which is related to the fractional Brownian motion.  For any
$H$ in $(0,1)$, the fractional Brownian motion of index (Hurst
parameter) $H$, $\{W_t^H;\ t\in[0,1]\}$ is the unique centered
Gaussian process whose covariance kernel is given by
  \begin{equation*}
    R_H(s,t)=\esp{W_s^HW_t^H}=
    \frac{V_H}{2}\Bigl( s^{2H}+t^{2H}-|t-s|^{2H}\Bigr)
  \end{equation*}
  where
\begin{equation*}
  V_H= \frac{\Gamma(2-2H)\cos (\pi H)}{\pi H (1-2H)}.
\end{equation*}
It has been proved (see \cite{decreusefond96.2}) that there exists a
standard Brownian motion such that almost surely~:
\begin{equation*}
  W_t^H=\int_0^t K_H(t,s)\, dB_s
\end{equation*}
where
\begin{equation}
  \label{defdekh}
  K_H(t,r)=\frac{(t-r)^{H- \frac{1}{2}}}{\Gamma(H+\frac{1}{2})}
F(\frac{1}{2}-H,H-\frac{1}{2}, H+\frac{1}{2},1- 
\frac{t}{r})1_{[0,t)}(r).
\end{equation}
The Gauss hyper-geometric function $F(\alpha,\beta,\gamma,z)$ (see
\cite{nikiforov88}) is the analytic continuation on ${\mathbb C}\times
{\mathbb C}\times {\mathbb C} \backslash \{-1,-2,\ldots \}\times
\{z\in {\mathbb C}, Arg |1-z| < \pi\}$ of the power series
 \begin{displaymath}
    \sum_{k=0}^{+ \infty} \frac{(\alpha)_k(\beta)_k}{(\gamma)_k 
k!}z^k.
 \end{displaymath}
 Here $(\alpha)_k$ denotes the Pochhammer symbol defined by
\begin{displaymath}
  (a)_0=1 \text{ and } (a)_k =
\frac{\Gamma(a+k)}{\Gamma(a)}=a(a+
1)\dots (a+k-1).
\end{displaymath}
It is well known from \cite{samko93} that $K_H$ maps continuously
$\L^p$ into $\I_{H+1/2,p}$ for any $p\ge 1,$ so that, in this case,
hypothesis (A) is fulfilled for any $\gamma < H,$ since
$\I_{H+1/2,1}\subset \I_{H-\epsilon, 2}$ for any $\epsilon$
sufficiently small.
\section{Existence and uniqueness of the solution of
  (\protect\ref{eq:sdefbm})}
\label{sec:exist-uniq-solut}

In the following, we denote by $c$ any irrelevant constant appearing
in the computations.
 
\begin{df}
  By a solution of the  equation (\ref{eq:sdefbm}), we
  mean a real-valued, progressively measurable stochastic process $
  X=\{X_t,\ t \in I\} $ such that $X$ belongs to $L^2(\Omega\times
  [0,1], \P\ox dt)$ and for any $t,$ $X_t$ is a.s. a  solution of (E).
\end{df}

\textsc{In the sequel, $r$ will denote a fixed real strictly greater than $\max(2,\gamma^{-1}).$}

\begin{thm} \label{existence} 
 Let $ b $ and $ \sigma $ be $L$--Lipschitz continuous with respect
  to their second variable, uniformly with respect to their first
  variable : For all t in $[0,1] $, for all x, y in $ {\mathbf R} $,
  \[ |b(t,x)-b(t,y)| + |\sigma(t,x)-\sigma(t,y)|
  \leq L |x-y|. \] Assume also that there exist $ x_0 \text{ and } y_0
  $ in $ {\rr}, $ such that $b(.,x_0)$ and $\sigma(.,y_0)$ belong to
  $\L^r.$ The differential equation (\ref{eq:sdefbm}) has then a
  unique continuous solution which belongs to
  $\Hol_{\gamma-1/r-\epsilon},$ for any $\epsilon \in (0,\gamma).$
\end{thm}
\begin{proof}
  The proof proceeds as usual by a fixed point technique. It is
  sufficient to note that for any $u$ and $v$ two progressively
  measurable processes belonging to $L^r(\Omega \x [0,1],\P \ox dt),$ we have~:
  \begin{enumerate}
  \item The process $\int_0^t K(t,s) u_s ds + \int_0^t K(t,s) v_s \,
    dB_s$ is continuous (according to (A)  and theorem
    \ref{thm:regularite}) and adapted, hence progressively measurable;
  \item since $\I_{\gamma+1/2,2}$ is continuously embedded in
    ${\mathcal B},$ according to hypothesis (A) , for any
    $1/r< \delta<\gamma,$ 
  \begin{align}
    \esp{\sup_{s\le t} | \int_0^t K(t,s) u_s \, ds|^r} & \le\esp{ \lVert
    \int_0^. K(.,s) u_s \, ds\rVert_{\I_{\delta,r}}^r}\label{eq:3}\\
&  \le \, c\,  \esp{ \int_0^t |u_s|^r \, ds};\label{eq:4}
  \end{align}
\item according to Theorem \ref{thm:regularite}, for any $1/r< \delta < \gamma,$
\begin{align}
  \esp{\sup_{s\le t} | \int_0^t K(t,s) v_s \, dB_s|^r} & \le c \esp{
    \lVert \int_0^. K(.,s) v_s \,
    dB_s\rVert_{\I_{\delta,r}}^r}\label{eq:5}\\
  & \le \, c\,  \lVert v \rVert^r_{L^r(\Omega\x [0,1])}.\label{eq:6}
  \end{align}
  \end{enumerate}
The uniqueness is then a consequence of (\ref{eq:4}),~(\ref{eq:6}) and the
Gronwall lemma.
According to (\ref{eq:4}),~(\ref{eq:6}), the Picard sequence defined
by~: 
\begin{equation*}
  X_t^0=x,\ X_t^n = x + \int_0^t K(t,s) b(s,X_s^{n-1})\, ds +\int_0^t
  K(t,s) \sigma(s,X_s^{n-1})\, dB_s
\end{equation*}
is a Cauchy sequence in $L^r(\Omega\x [0,1], \P \ox dt).$ We denote by
$X$ its limit. It is clearly a continuous and adapted solution of
(E). Furthermore, inequalities (\ref{eq:3}) and (\ref{eq:5}) entail
that the convergence also holds in $L^r(\Omega; \I_{\delta,r}),$ so
that the solution has a.s. $(\delta\!-\!1/r)$-H\"older continuous
sample-paths.
\end{proof}
\noindent \textbf{Remark : }
Note that in  the case of the fBm of Hurst index $H<1/2,$ we cannot
work as usual in $L^2$
but only in $L^{1/H},$ hence we have to ensure stronger regularity on
the coefficients, i.e., $b(.,x_0)$ and $\sigma(.,x_0)$ must  belong
to $\L^{1/H}.$

By the same techniques, we can prove that~:
\begin{thm}
  Under the hypothesis of the previous theorem, the map which sends $x$
  to the solution of (E) with initial condition $x$ is continuous from
  $\rr$ in $\L^r(\Omega\x [0,1], \P\ox dt).$ 
\end{thm}

\section{Gross-Sobolev regularity of $X$}
We are now interested in the Gross-Sobolev differentiability of $X_t.$ 
\begin{lemma}
\label{derivint}
  For $u$ adapted belonging to $\L^2(\Omega\x [0,1]; \D_{2,1}),$ for any $t\in
  [0,1],$ the  distribution  $ \nabla(\int_0^t K(t,s) u_s \, dB_s)$
  exists as a   $L^{2}(\Omega\x [0,t])$ random
  variable and satisfies~:
\begin{equation}
\label{eq:9}
   \lVert \nabla(\int_0^t K(t,s) u_s \,
    ds)\rVert_{L^2(\Omega\x [0,t])}  \le \, c \, \lVert u
    \rVert_{L^2(\Omega\x [0,t])}
\end{equation}
\end{lemma}
\begin{lemma}
\label{derivintsto}
  For $u$ bounded, adapted belonging to $\L^2(\Omega\x [0,1]; \D_{2,1}),$ for any $t\in
  [0,1],$ the  distribution  $  \nabla(\int_0^t K(t,s)u_s \, dB_s )$ exists as
  a $L^{2}(\Omega\x [0,t])$ random
  variable and satisfies~:
\begin{equation}
\label{eq:2}
 \lVert \nabla(\int_0^t K(t,s) u_s \,
  dB_s)\rVert_{L^2(\Omega\x [0,t])}  \le \, c \, \lVert  u
  \rVert_{L^2(\Omega\x [0,t]; \D_{2,1})}.
\end{equation}
\end{lemma}
\begin{proof}
  We only prove (\ref{eq:2}) since inequality (\ref{eq:9}) is simpler to
  show and its proof proceeds along the same lines. According to (A), for any
  $t,$ the random variable $\int_0^t K(t,s) u_s \, dB_s$ belongs to
  $L^r(\Omega, \P)$ and thus has a derivative in the distributional sense.  For
  any $\xi \in \D_\infty(\L^2),$
  \begin{multline*}
    \langle \nabla(\int_0^t K(t,s) u_s \, dB_s), \, \xi \rangle_{
      \D_{-\infty}, \D_\infty} = \esp{\int_0^t K(t,s) u_s \, dB_s\, 
      \delta(\xi)}\\
  \begin{aligned}
    = & \esp{\int_0^t K(t,s) u_s \, \xi_s \, ds} +\esp{\int_0^t K(t,s)
      \Bigl(\int_0^t \nabla_r u_s \, \xi_r \, dr\Bigr)\, dB_s}.
  \end{aligned}
  \end{multline*}
Hypothesis (A) induces that for any $\xi\in
\D_\infty(\L^{2}),$ we have~:
\begin{align*}
\esp{ |\int_0^t K(t,s) u_s \, \xi_s \, ds| }& \le \, c\,  \lVert K(t,.)u
  \rVert_{L^2(\Omega \x [0,1])} \lVert \xi \rVert_{L^2(\Omega \x [0,1])}\\
& \le \, c\, \lVert u \rVert_\infty \lVert K^*(\epsilon_t) \rVert_{\L^2} \lVert \xi \rVert_{L^2(\Omega \x [0,1])}.
\end{align*}
Since $\epsilon_t$ belongs to the dual of $\I_{\gamma+1/2,2},$
according to (A), we have~:
\begin{equation}
\label{eq:10}
   \lVert K^*(\epsilon_t) \rVert_{\L^2}\le \, c\, \lVert \epsilon_t
   \rVert_{\I_{\gamma+1/2,2}^*}=c t^\gamma.
\end{equation}
According to theorem \ref{thm:regularite} for $r=2$, since
$\I_{\gamma-\epsilon,2}\subset \L^2,$ for any $\xi\in
\D_\infty(\L^{2}),$ 
\begin{multline}\label{eq:8}
  \esp{|\int_0^t \xi_\tau \int_0^t K(t,s)
       \nabla_\tau u_s  \, dB_s\, d\tau|}  \\
 \le \, c \, \lVert \xi \rVert_{L^{2}(\Omega\x
       [0,1])}\esp{\int_0^t\int_0^s |\nabla_\tau  u_s|^2 \, d\tau \, ds}^{1/2} .
\end{multline}
The result is then a consequence of (\ref{eq:10}) and
(\ref{eq:8}) and Proposition 3 of \cite[page 37]{ustunel_book}.
\end{proof}
\noindent \textbf{ Remark : } If $\gamma >1/2,$ inequalities
(\ref{eq:9}) and (\ref{eq:2}) are true uniformly with respect to $t,$
i.e., for instance,
\begin{equation*}
 \lVert\sup_{t\le T} \nabla(\int_0^t K(t,s) u_s \,
  dB_s)\rVert_{L^2(\Omega\x [0,T])}  \le \, c \, \lVert  u
  \rVert_{L^2(\Omega\x [0,T]; \D_{2,1})}.
\end{equation*}

\begin{thm}
\label{derivabilitedeX}
The hypothesis of Theorem \ref{existence} are assumed to hold.
Moreover, $b$ and $\sigma$ are supposed to be once continuously
differentiable with respect to their space variable, with bounded
derivative; assume furthermore that $\sigma$ is bounded.  For any
$t\in I,$ the value at $t$ of the solution of Eqn.  \eqref{eq:sdefbm},
denoted by $X_t,$ belongs to $\D_{2,1}.$ For any $\xi\in
\H,$
\begin{multline}
\label{eq:11}
  < X_t,\, \xi>_\H =\int_0^t
  K(t,s)\sigma(X_s)\xi_s\, ds \\
  \begin{aligned}
  + &\int_0^t
  K(t,u) \dfrac{\partial b}{\partial x}(u,X_u) <\n X_u,\, \xi>_\H\, du \\
  + & \int_0^t K(t,u) \dfrac{\partial \sigma}{\partial x}(u,X_u)
  <\n X_u,\, \xi>_\H \, dB_u.   
  \end{aligned}
 \end{multline}
Moreover, for $\xi\in L^r(\Omega \x [0,1]),$ $(<\n X_t,\, \xi>_\H,\, t\in [0,1])$ belongs to
$L^r(\Omega\x [0,1]).$
\end{thm}
\begin{proof}
  Let $X^n$ be the Picard sequence already defined in the proof of
  theorem \ref{existence}. Using lemmas \ref{derivint} and
  \ref{derivintsto}, we prove by induction on $n$ that $X^n$ belongs
  to $\D_{2,1}$ and that~:
  \begin{equation*}
    \lVert X_t^n  \rVert_{\D_{2,1}}^2 \le \, c\, \int_0^t \lVert
    X_s^{n-1}  \rVert_{\D_{2,1}}^2 \, ds \le \frac{(ct)^n}{n!}x^2. 
  \end{equation*}
It follows that $\sup_n \lVert X_t^n
\rVert_{\D_{2,1}}$ is finite and thus that there exists a weakly
convergent subsequence in $\D_{2,1}.$  Since $X^n_t$ converges to
$X_t$ in $L^2(\Omega),$ the closability of $\n$ entails that $X_t$
belongs to $\D_{2,1}.$ Since for any $n,$ 
\begin{multline*}
  < X_t^n,\, \xi>_\H =\int_0^t
  K(t,s)\sigma(X_s^{n-1})\xi_s\, ds \\
  \begin{aligned}
  + &\int_0^t
  K(t,u) \dfrac{\partial b}{\partial x}(u,X_u^{n-1}) <\n X_u^{n-1},\, \xi>_\H\, du \\
  + & \int_0^t K(t,u) \dfrac{\partial \sigma}{\partial x}(u,X_u^{n-1})
  <\n X_u^{n-1},\, \xi>_\H \, dB_u,   
  \end{aligned}
 \end{multline*}
a straightforward application of the dominated convergence theorem
yields to (\ref{eq:11}).
Since $\L^r$ is continuously imbedded in $\L^2$ and
$\I_{\gamma+1/2,2}$ is continuously embedded in ${\mathcal
  C}_0([0,1];\rr),$ we have~:
\begin{multline*}
  \esp{\sup_{s\le t }| <\n X_s ,\, \xi>_\H|^r}  \le c\Bigl( \esp{\int_0^1 |\xi_s|^r\,
  ds}+ \esp{\int_0^t |<\n X_s,
  \xi>_\H|^r \, ds}\Bigr)\\
\le  c\Bigl( \esp{\int_0^1 |\xi_s|^r\,
  ds}+ \esp{\int_0^t \sup_{u\le s}|<\n X_u,
  \xi>_\H|^r \, ds}\Bigr).
\end{multline*}
By Gronwall lemma, it follows that $ <\n X_. ,\, \xi>_\H$ belongs to
$L^r(\Omega\x [0,1]).$ 
\end{proof}
\begin{thm}
\label{thmequationdenabla}
Assume that the hypothesis of Theorem \ref{derivabilitedeX} hold.  For any
$\xi\in L^r(\Omega\x [0,1]),$ the equation
  \begin{multline}
    Y_t=<K(t,.)\sigma\circ X, \, \xi>_\H  +\int_0^t K(t,u) \frac{\partial b}{\partial x}(u,X_u) Y_u\, du \\
    + \int_0^t K(t,u)\frac{\partial \sigma}{\partial x}(u,X_u) Y_u \,
    dB_u
\label{equationdenabla}
  \end{multline}
  has one and only one solution belonging to $L^r(\Omega\x [0,1]).$ 
\end{thm}
\begin{proof}
  Theorem \ref{derivabilitedeX} stands that the process
  $(\langle\n X_t,\, \xi\rangle_\H,\, t\in [0,1])$ is a  solution
  of \eqref{equationdenabla} with the desired integrability property.

If $Y$ and $Z$ are two such solutions, according to hypothesis (A) and
to theorem \ref{thm:regularite}, we have~: 
\begin{equation*}
  \esp{|Y_t-Z_t|^r}\le c \int_0^t \esp{|Y_s-Z_s|^r}\, ds.
\end{equation*}
By iteration, this induces that $Y=Z,$ $\P\ox dt$ almost everywhere. 
\end{proof}
\begin{thm}
  Assume that the hypothesis of theorem \ref{derivabilitedeX} hold.  
  Let $V_0(t,s)$ be a measurable deterministic kernel such that~:
  \begin{equation}
    \label{eq:7}
    \int_0^1\!\!\int_0^1 |V_0(u,s)|^r \, du\, ds < \infty.
  \end{equation}
 For   $n\ge 1,$ consider
   \begin{multline*}
     V_{n+1}(t,s) = \int_s^t K(t,u) \frac{\partial b}{\partial
       x}(u,X_u) V_n(u,s)\, du \\ + \int_s^t K(t,u) \frac{\partial
       \sigma}{\partial x}(u,X_u) V_n(u,s) \, dB_u.
\end{multline*}
The two following  properties hold~:
\begin{description}
\item[V1] $L(t,s)=\sum_{n=0}^{+\infty} V_n(t,s)$ is a convergent
  series in $L^{r}(\Omega\x [0,1]^2).$ 
\item[V2] $L(t,s)$ is a solution of
        \begin{multline}
\label{eq2param}
L(t,s)-V_0(t,s)=\int_s^t K(t,u) \frac{\partial b}{\partial x}(u,X_u) L(u,s)\, du \\
+ \int_s^t K(t,u) \frac{\partial \sigma}{\partial x}(u,X_u) L(u,s) \,
dB_u.
  \end{multline}
\end{description}
  \end{thm}
  \begin{proof}
    By the techniques used above, we can show that~:
    \begin{equation*}
\int_0^\zeta \int_0^1 |V_{n+1}(t,s)|^r \, dt \, ds \le \, c \,
\int_0^\zeta \int_0^1 \int_s^t |V_n(u,s)|^r \, du \, ds \, dt,
    \end{equation*}
hence if we set 
$$\psi_n(t)=\int_0^t \! \! \int_0^1 |V_n(u,s)|^r\, ds \, du,$$
the previous equation reads as 
\begin{equation*}
  \psi_{n+1}(\zeta)\le \, c \, \int_0^\zeta \psi_n(t)\, dt.
\end{equation*}
Applying Gronwall lemma and (\ref{eq:7}), it follows that the series
$\sum_n V_n$ is convergent in $L^r(\Omega\x [0,1]^2).$ Moreover, it is
straightforward that~:
 \begin{multline*}
   \int_s^t K(t,u) \frac{\partial b}{\partial x}(u,X_u) \sum_{j=0}^ n
   V_j(u,s)\, du \\ + \int_s^t K(t,u) \frac{\partial \sigma}{\partial
     x}(u,X_u) \sum_{j=0}^ n V_j(u,s) \, dB_u = \sum_{j=0}^ {n+1}
   V_j(t,s)-V_0(t,s).
 \end{multline*}
 According to hypothesis (A) and theorem \ref{thm:regularite}, the
 left-hand-side of the last equation converges to the right-hand-side
 of (\ref{eq2param}).
  \end{proof}
\noindent \textbf{\textsc{Hypothesis} (B)--}\textsl{
  We assume that there exist $g$ an almost surely positive function
  such that $V_0(t,s)=K(t,s)g(s)$ satisfies (\ref{eq:7}).}

\begin{thm}[Parameter variation formula]
\label{parametervariation}
Assume that all the hypothesis made so far  hold.  Consider
the space $\HH$ of elements of $\H$ satisfying $\xi g^{-1}\in \L^{r}.$
 Let $L$ be defined as in the previous theorem with the value of
$V_0$ taken in hypothesis (B). For any $\xi \in
\HH,$ let
\begin{math}
   Y_t = \int_0^t L(t,s) \sigma(X_{s})g^{-1}(s)\xi(s) \, ds.
\end{math}
For any $t\in [0,1],$ we have
\begin{math}
<\n X_t,\, \xi>_\H=  Y_{t}, \ \P
\end{math}
 almost surely.
\end{thm}
\begin{proof}
  For $\xi \in \HH,$  since $\sigma$ is bounded, $Y$ belongs to $L^r(\Omega\x
  [0,1]).$ It is clear that $Y$ is a formal solution of
  (\ref{equationdenabla}) hence by theorem \ref{thmequationdenabla},
  the equality $<\n X_t,\, \xi>_\H=  Y_{t}$ follows.
\end{proof}

\noindent  \textbf{Remark : } For the fractional Brownian motion, it
  is proved in \cite{decreusefond96.2} that 
  \begin{equation*}
    0\le K_H(t,s) \le \, c\, (t-s)^{H-1/2}s^{-|H-1/2|},
  \end{equation*}
  hence hypothesis (B) is satisfied with $g(s)=s^{|H-1/2|}.$ It is a little
  counter-intuitive that as $H$ increases towards $1,$ hypothesis (B)
  requires an increasing value of $\nu$ to be fulfilled. It is due to the
  increasing singularity of $K_H(t,s)$ whereas $K_H$
  as a map is more and more regularizing.

\appendix
\section{Deterministic fractional calculus}
\label{sec:determ-fract-calc}
For $f\in \L^1([0,1])$, the left and right fractional integrals of $f$
are defined by~:
        \begin{align*}
          (I_{0^+}^{\alpha}f)(x) & =
          \frac{1}{\Gamma(\alpha)}\int_0^xf(t)(x-t)^{\alpha-1}dt\ ,\ 
           x \ge 0,\\
          (I_{b^-}^{\alpha}f)(x) & =
          \frac{1}{\Gamma(\alpha)}\int_x^bf(t)(t-x)^{\alpha-1}dt\ ,\ 
          x\le b,
        \end{align*}
        where $\alpha>0$ and $I^0=\Id.$ In what follows, $T $ is a
         real in $(0,1].$ For any $\alpha\ge 0$, any
        $f\in \L^p([0,T])$ and $g\in \L^q([0,T])$ where
        $p^{-1}+q^{-1}\le \alpha$, we have~:
\begin{equation}
  \label{int_parties_frac}
  \int_0^T f(s)(I_{0^+}^\alpha g)(s)\ ds = \int_0^T (I_{T^-}^\alpha 
f)(s)g(s)\ ds.
\end{equation}
The Besov space $I^\alpha_{0^+}(\L^p([0,T]))= \I_{\alpha,p}([0,T])$ is
usually equipped with the norm~:
\begin{equation*}
  \| f \| _{ \I_{\alpha,p}}=\| I^{-\alpha}_{0^+} f\|_{\L^p([0,T])}.
\end{equation*}
We then have the following continuity results (see
\cite{feyel96,samko93})~:
\begin{prop}\label{continuiteintfrac}~\\
For each $0< T\le 1,$
  \begin{enumerate}
  \item \label{inclusionLpLq} If $0<\alpha <1,$ $1< p <1/\alpha,$ then
    $I^\alpha_{0^+}$ is a bounded operator from $\L^p([0,T])$ into
    $\L^q([0,T])$ with $q=p(1-\alpha p)^{-1}.$
  \item \label{thm:inclusions} For any $0< \alpha <1$ and any $p\ge
    1,$ $\I_{\alpha,p}([0,T])$ is continuously embedded in $\Hol_{\alpha-
    1/p}([0,T])$ provided that $\alpha-1/p>0.$ $\Hol_{\nu}([0,T])$ denotes the space
    of H\"older--continuous functions, null at time $0,$ equipped with
    the usual norm~:
    \begin{equation*}
      \lVert f \rVert_{\Hol_{\nu}([0,T])}=\sup\limits_{0\le t\neq s\le T} \frac{| f(t)-f(s)| }{|t-s|^\nu}.
    \end{equation*}
  \end{enumerate}
\end{prop}
By $I^{-\alpha}_{0^+},$ respectively $I^{-\alpha}_{1^-},$ we mean the
inverse map of $I^{\alpha}_{0^+},$ respectively $I^{\alpha}_{1^-}.$
When we don't precise the interval $[0,T]$ in the notations of $\L^p$
spaces or of Besov spaces, it is meant that $T=1.$
\section{Malliavin calculus}
\label{sec:malliavin}
We only give the few results we need, further details can be found in
   \cite{malliavin_book,ustunel_book}.
We work on the standard Wiener space $(\Omega,{\H},\P)$ where $\Omega$ is the
Banach space of continuous functions from $[0,1]$ into $\rr,$ null
at time $0,$ equipped with the sup-norm. ${\H}$ is the Hilbert space of
absolutely continuous function with the norm $\| h \|_{{\H}}= \| \dot{h}
\|_{\L^2},$ where $\dot{h}$ is the time derivative of $h.$ A
mapping $\phi$ from $\Omega$ into some separable Hilbert space $X$ is
called cylindrical if it is of the form
$\phi(w)=f(<v_1,w>,\cdots,<v_n,w>)$ where $f\in C_0^\infty (\rr^n,
X)$ and $v_i\in \Omega^\star$ for $i=1,\cdots,n$. For such a function we
define $\n \phi$ as
$$
\n\phi(w)=\sum_{i=1}^n \partial_i
f(<v_1,w>,\cdots,<v_n,w>){\tilde{v}}_i\,,
$$
where $\tilde{v}_i$ is the image of $v_i$ under the injection
$\Omega^\star\hookrightarrow \L^2$.  From the quasi-invariance of the Wiener
measure, it follows that $\n$ is a closable operator on $L^p(\Omega;X)$,
$p\geq 1$, and we will denote its closure with the same notation. The
powers of $\n$ are defined by iterating this procedure. For $p>1$,
$k\in \nn$, we denote by $\D_{p,k}(X)$ the completion of $X$-valued
cylindrical functions under the following norm
$$
\|\phi\|_{p,k}=\sum_{i=0}^k \|\n^i\phi\|_{L^p(\Omega;X\otimes (\L^2)^{\otimes
    i})}\,.
$$
Let us denote by $\delta$ the formal adjoint of $\n$ with respect
to Wiener measure, a classical result stands that $\delta$ is an
extension of the It\^o integral thus we have
\begin{equation}
\label{eq:partiesmalliavin}
  \esp{\int_0^t  u_s \, dB_s\  \varphi}= \esp{\int_0^t u_s
    \nabla_s \varphi \, ds}
\end{equation}
for any $u$ adapted in $L^2(\Omega;\L^2)$ and any $\varphi \in \D_{2,1},$
where $\{B_t= \delta({\mathbf 1}_{[0,t]}), \, t\in [0,1]\}$ is a
standard Brownian motion on $(\Omega,\P).$ 

\end{document}